\documentclass[12pt]{article}

\usepackage{amsmath,
amsfonts,latexsym,graphicx,amssymb,amsthm}

\usepackage{color}
\usepackage{authblk}
\usepackage[new]{old-arrows}

\newcommand{\C}{\mathbb{C}}
\newcommand{\F}{\mathbb{F}}

\renewcommand{\P}{\mathbb{P}}
\newcommand{\Q}{\mathbb{Q}}
\newcommand{\R}{\mathbb{R}}
\newcommand{\Z}{\mathbb{Z}}

\newcommand{\Aff}{\mathrm{Aff}}
\newcommand{\Aut}{\mathrm{Aut}}
\newcommand{\Card}{\mathrm{Card\,}}

\renewcommand{\d}{\mathrm{d}}

\newcommand{\End}{\mathrm{End}}
\newcommand{\Elem}{\mathrm{Elem}}

\newcommand{\GL}{{\mathrm {GL}}}

\newcommand{\id}{\mathrm{id}}

\newcommand{\Image}{\mathrm{Im}}

\newcommand{\Ker}{\mathrm{Ker}}

\newcommand{\Spec}{{\mathrm {Spec}}\,}

\renewcommand{\mod}{\,\mathrm{mod}\,}

\newtheorem*{MainA}{Theorem A}
\newtheorem*{MainB}{Theorem B}
\newtheorem*{MainC}{Theorem C}

\newtheorem*{MainA2}{Theorem A.2}
\newtheorem*{MainC1}{Theorem C.1}
\newtheorem*{MainC2}{Theorem C.2}

\newtheorem*{Cor1.1}{Corollary 1.1}
\newtheorem*{Cor1.2}{Corollary 1.2}
\newtheorem*{CorD}{Corollary D}
\newtheorem*{CorD1}{Corollary D.1}
\newtheorem*{CorD2}{Corollary D.2}

\newtheorem*{vdKthm}{van der Kulk Theorem}
\newtheorem*{Cornu}{Cornulier Theorem}

\newtheorem*{proposition}{Proposition}

\newtheorem{lemma}{Lemma}

\newtheorem*{quest}{Question}

\newcommand{\ch}{\operatorname{ch}}
\newcommand{\tr}{\operatorname{tr}}

\font\small=cmr10

\begin{document}

\title{Linearity and Nonlinearity of groups of polynomial automorphisms of $K^2$}

\author{Olivier Mathieu}

\affil{\small Institut Camille Jordan du CNRS\\

\small Universit\'e de Lyon\\ 

\small F-69622 Villeurbanne Cedex\\

\small mathieu@math.univ-lyon1.fr}


\maketitle

\begin{abstract} Let $K$ be a field, and let $\Aut \,K^2$
be the group of polynomial automorphisms of $K^2$. We investigate
which subgroups are linear or not. In characteristic zero,
there are small nonlinear subgroups and some big linear subgroups.
When $K$ has finite characteristic, the whole group
$\Aut\,K^2$ is linear whenever $K$ is finite, and nonlinear otherwise.
\footnote{Research supported by UMR 5208 du CNRS}

\end{abstract}

\noindent
\centerline{\it This paper is respectfully dedicated to 
Jacques Tits.}

\section*{Introduction}

Recall that a group $\Gamma$ is called {\it linear}, 
or {\it linear over a ring} in case of ambiguity, if
there is an embedding $\Gamma\subset GL(n,R)$
for some integer $n$ and some commutative ring $R$. 
Moreover $\Gamma$ is called {\it linear over a field} 
if it can be embedded into $GL(n,K)$ for some integer $n$ and some field $K$.

Let $\Aut\,K^2$ be the group of polynomial automorphisms  
of the affine plane $K^2$. In this paper, we will investigate the linearity or nonlinearity properties for the subgroups 
of $\Aut\,K^2$. In particular, we will  consider the following subgroups

\centerline{$\Aut_0\,K^2=\{\phi\in \Aut\,K^2\vert\,\phi({\bf 0})={\bf 0})\}$, and }

\centerline{$\Aut_1\,K^2=\{\phi\in \Aut_0\,K^2\vert\,{\textnormal d}
\phi\vert_{\bf 0}=\id$\}.}

\noindent The first result of the present paper is

\begin{MainA} (A.1) If $K$ is infinite, the group $\Aut_0\,K^2$ is not linear, even over a ring.

(A.2) Moreover if $\ch\,K=0$
the group $\Aut_0\,K^2$ contains finitely generated subgroups which are  not linear,
even over a ring.
\end{MainA}

It was known that the much larger Cremona group $Cr_2(\Q)$ is
not linear over a field, see \cite{C} \cite{P}. 
More recently, the nonlinearity over a field of the whole group
$\Aut\,\Q^2$ was proved in \cite {Co}. In the same paper, Y. Cornulier raises the question (answered by Theorem A) 
of finding a nonlinear finitely generated (FG in the sequel) 
subgroup in $\Aut\,\Q^2$. 

Various authors show that automorphism groups share some properties with linear groups, see e.g. 
\cite{S00}\cite {BPZ}. More specifically, it was proved in  \cite{L} that
$\Aut\,K^2$ satisfies Tits alternative. Theorem A shows that these results are not a consequence of  classical 
results for linear groups. 

Roughly speaking, Theorem A.2 means that $\Aut_0\,K^2$ contains "small" subgroups 
(indeed FG subgroups) which are not linear, at least if $\ch\,K=0$. On the 
opposite there is

\begin{MainB}
For any  field $K$, the group $\Aut_1\,K^2$ is
 linear over the field $K(t)$.
\end{MainB}

\noindent 
Indeed a stronger result is proved: under 
a mild hypothesis over $K$, the group $\Aut_1\,K^2$ is
 linear over the field $K$ itself, see Corollary D in
 Section 9.

Following Demazure's  approach \cite{D}, 
the functor 
 $K\mapsto \Aut\,K^2/\Aut_1\,K^2$ is represented by
a $6$-dimensional variety,  namely the affine 
group of $K^2$. Therefore Theorem B states that 
$\Aut\,K^2$ contains some "big" subgroups (of finite codimension) which are linear over a field.

Assume again that $K$ is infinite. In view of Theorems A and B
we can ask which  groups between $\Aut_1\,K^2$  and $\Aut_0\,K^2$  are linear. These groups are defined as

\centerline{$\Aut_S\,K^2:=\{\phi\in\Aut_0\,K^2\,\vert\, \d\phi_{\bf 0}\in S\}$,}

\noindent where $S$ is a subgroup of $GL(K^2)$.

An element $g\in GL(K^2)$ is called {\it $K$-reducible} if its eigenvalues lies in $K$. The next result provides a partial answer to the previous question

\begin{MainC} Let $S$ be a subgroup of $GL(K^2)$. Assume that
$S$ satisfies one of the following assertions

(i) Any $K$-reducible  $g\in S$ is  unipotent, or 

(ii) the group $S$ is FG, and any $K$-reducible  $g\in S$ is 
 quasi-unipotent.

Then the group $\Aut_S\,K^2$ is linear over $K(t)$.
\end{MainC}

For example, $\pm 1$ are the only two reducible elements of 
$SO(2,\R)$. Therefore  $\Aut_{S}\,\R^2$ is linear over $\R(t)$
for any FG subgroup $S$ of $SO(2,\R)$. We believe that 
the whole group $\Aut_{SO(2,\R)}\,\R^2$ is linear and more generally

\begin{quest} Let $S$ be a subgroup of $GL(K^2)$. Is the group
$\Aut_S\,K^2$  linear if and only if there is an integer $N$
such that  $g^N$ is unipotent for any $K$-reducible 
$g\in S$?
\end{quest}

As a consequence of Theorem B we have

\begin{CorD} (D.1) Let $K$ be a finite field. Then
$\Aut\,K^2$ is linear over $K(t)$.

(D.2) Let $p$ be a prime number. Then any FG subgroup of
$\Aut \overline{\F}_p$ is linear over $\F_p(t)$.

\end{CorD}

Note that Corollary D.1 treats the case not covered by Theorem A.1.
Moreover, the question of the existence of nonlinear
FG subgroups of $\Aut\,K^2$ is solved by Theorem A.2 and Corollary D.2, except for
finite characteristic fields of transcendence degree $\geq 1$,
therefore we ask the following

\noindent {\bf Question.} {\it Let $K$ be a field
containing $\F_p(t)$. Does $\Aut\,K^2$ contains
nonlinear FG subgroups?}

The paper is organized as follows. In section 1,
the general notions concerning $\Aut\,K^2$ are defined,
and the classical result of van der Kulk \cite{vdK} 
is stated. In the next section, it is proved that
some linear groups obtained by amalgamation  are indeed
linear over a field. Therefore, it is possible to use
the theory of algebraic groups to show that some groups are not
linear. This is used to prove Theorem A.2 in section 3.

Borel and Tits \cite{BT} showed that certain group morphisms of
algebraic groups are semi-algebraic. A very simple form of their ideas are used to prove Theorem A.1, see Sections 4 and 5.

Next the proof of Theorem B and C is based one some  Ping-Pong ideas. These ideas were originally invented for the dynamic of groups with respect to 
the euclidean metric topologies \cite{FN}, but  they were used by Tits in the context of the ultrametric topologies \cite{T72}.

\bigskip\noindent
{\it Aknowledgements} J.P. Furter and R. Boutonnet informed us that they independently found a FG subgroup of $\Aut\, \Q^2$
which is not linear over a field \cite{BF}. We also heartily thank
S. Lamy for interesting comments and E. Zelmanov for an inspiring talk.

 \section{The van der Kulk Theorem}

In this section, we review the basic facts about the amalgamated products of groups. Then the we recall the classical 
van der Kulk Theorem.

\bigskip
\noindent {\it 1.1  Amalgamated products}
 
\noindent  Let $A$, $G_1$ and $G_2$ be groups. Their {\it free product} is denoted by $G_1 *G_2$, see \cite{MKS} ch.4.  Now let 
$A {\buildrel f_1\over\rightarrow} G_1$ and
$A {\buildrel f_2\over\rightarrow} G_2$ be two injective group morphisms and let $K\subset G_1 *G_2$ be the invariant subgroup
generated by the elements $f_1(a)f_2(a^{-1})$ when $a$ runs
over $A$. By definition the group $G_1*G_2/K$ is called 
the {\it amalgamated product} of $G$ and $H$ over $A$,
and it is denoted by $G_1 *_A G_2$, 
see e.g. \cite{S83}, ch. I. 
 
 In the literature, the amalgamated products are also called
free amalgamated products, see \cite {MKS} ch.8. 

Recall that the natural maps 
 $G_1\rightarrow\Gamma$ and $G_2\rightarrow\Gamma$ are
 injective, see the remark after the Theorem 1 of  ch. 1 in \cite{S83}. Hence, we will use a less formal terminology.
The group $A$ will be viewed as a common subgroup 
 of $G_1$ and $G_2$, and $G_1$ and $G_2$ will be viewed as
 subgroups of  $G_1 *_A G_2$.

  \bigskip \noindent 
{\it 1.2 Reduced words}

\noindent
The usual definition \cite{S83} of reduced words
is based on the right $A$-cosets. In order to avoid a confusion
between the set difference notation $A\setminus X$ 
and the $A$-orbits notation $A\backslash X$, we will
use a definition based on the left $A$-cosets.

Let $G_1$, $G_2$ be two groups sharing a common subgroup $A$,
and let $\Gamma=G_1 *_A G_2$.
Set $G_1^*=G_1\setminus A$,  $G_2^*=G_2\setminus A$
and let $T_1^*\subset G_1^* $ (respectively 
$T_2^*\subset G_2^* $) be a set of  representatives of  
$G_1^*/A$ (respectively of $G_2^*/A$).
Let $S$ be the set of finite alternative sequences of one's and two's, and let $\epsilon=(\epsilon_1,\dots,\epsilon_n)\in S$.

A {\it reduced word}  is a word
$(x_1,\dots x_n,x_0)$ where $x_0$ is in $A$, and
$x_i\in G_{\epsilon_i}^*$ for $i\geq 1$,
where the sequence $\epsilon=(\epsilon_1,\dots,\epsilon_n)$ 
lies in $S$. Moreover this sequence $\epsilon$ is called the
{\it  type
of the reduced word} $( x_1,\dots x_n,x_0)$.
Let ${\cal R}$ be the set of all reduced words.
The next lemma is well-known, see  
 e.g. \cite{S83}, Theorem 1.

\begin{lemma}\label{words} The map

\centerline{ $(x_1,\dots x_n,x_0)\in {\cal R} 
 \mapsto  x_1\dots x_n x_0 \in G_1*_A G_2$}
 
 \noindent is bijective.
 
 \end{lemma}
 
 Set $\Gamma=G_1*_A G_2$. For $\gamma\in \Gamma\setminus A$
 there is some $n\geq 1$, some 
 $\epsilon\in S$ and some $g_i\in G_{\epsilon_i}^*$
 such that $\gamma=g_1...g_n$. It follows from
 the remark of \cite {S83} (after Theorem 1, ch. 1) that
 $\gamma=x_1...x_nx_0$ for some reduced word 
 $(x_1...x_nx_0)$ of type $\epsilon$. Since $\epsilon$ and $n$ are determined by $\gamma$, the sequence $\epsilon$ is called 
{\it the type} of $\gamma$ and 
the integer $l(\gamma):=n$ is called it's {\it length}.

 \bigskip \noindent 
{\it 1.3 Amalgamated product of subgroups}

\noindent
Let $G_1$, $G_2$ be two groups sharing a common subgroup $A$.
Let $G_1'\subset G_1$ and $G_2'\subset G_2$ be subgroups, with the property that $G_1'\cap A=G'_2\cap A$. Set 
$A':=G_1'\cap A=G'_2\cap A$.

 \begin{lemma}\label{subamal} With the previous hypotheses,
 the natural map 
 
 \centerline{$G_1'*_{A'} G_2'\mapsto G_1*_A G_2$}
 
 \noindent is injective. Moreover if the 
 the natural maps $G'_1/A'\to G/A$ and $G_2'/A'\to G_2/A$ are bijective, then we have
 
\centerline{$G_1*_A G_2/G_1'*_{A'} G_2'\simeq A/A'$.}

 \end{lemma}

\begin{proof}  For $i=1,2$ set $G_i^*=G_i\setminus A$,
$G_i^{'*}=G'_i\setminus A'$. Let 
$T^*_i\subset G_i^*$ (respectively $T^{'*}_i\subset G_i^{'*}$)
be a set of representatives of
$G_i^*/A$ (respectively of 
$G_i^{'*}/A'$).

Since the maps
$G_i'/A'\to G_i/A$ are injective, it can be assumed that  $T^{'*}_i\subset T^{*}_i$. Let $\cal R$ 
(respectively $\cal R'$) be the set of reduced words of
$G_1*_A G_2$ (respectively of $G_1'*_{A'} G_2'$).
By definition, we have $\cal R'\subset \cal R$, thus
by lemma \ref{subamal} the map $G_1'*_{A'} G_2'\mapsto G_1*_A G_2$
is injective.

Moreover, assume that the  maps $G'_1/A'\to G/A$ and $G_2'/A'\to G_2/A$ are bijective. It follows that ${\cal R'}$ is the set of
all reduced words $(x_1,\dot,x_n,x_0)\in{\cal R}$ such that
$x_0\in A'$. It follows easily that

\centerline{$G_1*_A G_2/G_1'*_{A'} G_2'\simeq A/A'$.}

\end{proof}

 \bigskip\noindent
 {\it 1.4 The group $\Aut\,K^2$} 
 
 \noindent Let $K$ be a field.
 Recall that $\Aut\,K^2$ is the group of polynomial automorphisms of
 $K^2$, so an element of $\Aut\,K^2$ is a polynomial map
 $\phi:(x,y)\mapsto (f(x,y), g(x,y))$, where $f,\,g\in K[x,y]$ which admits a polynomial inverse. When $\ch\,K=0$, for the existence of a polynomial inverse it is enough that the naive map
 $\phi: K^2\to K^2$  is bijective. When $\ch\,K\neq 0$, this is not enough. Even when $K$ is infinite, the existence
 of a polynomial inverse requires some additional
separability condition.
 
 \bigskip\noindent
 {\it 1.5 The subgroup $\Elem(K^2)$ of elementary automorphism of $K^2$}
 
\noindent By definition, an {\it elementary automorphism}  of $K^2$ is an automorphism $\phi\in \Aut\,K^2$ of the form 
 
 \centerline{$\phi:(x,y)\mapsto (z_1 x +t, z_2 y +f(x))$}
 
 \noindent for some $z_1, z_2\in K^*$, some $t\in K$ and some $f\in K[x]$. The group of elementary automorphism is denoted
 $\Elem (K^2)$. Also set
 
 \centerline{$\Elem_0 (K^2)=\Elem (K^2)\cap \Aut_0\,K^2$, and}
 
 \centerline{$\Elem_1 (K^2)=\Elem (K^2)\cap \Aut_1\,K^2$.}

\noindent For any $F\in K[t]$, let $\mu(F)$ be the automorphism

\centerline{ $\mu(F):(x,y)\mapsto (x,y+x^2 F(x))$.}

\noindent Clearly $\mu:K[t]\to \Elem_1 (K^2)$ is group isomorphism.

 \bigskip\noindent
 {\it 1.6 The affine group $\Aff(K^2)$}

\noindent Let $\Aff(K^2)\subset \Aut\,K^2$ 
(respectively $\GL(K^2)\subset \Aut\,K^2$) be the subgroup
of affine automorphisms (respectively linear automorphisms)  of $K^2$.  Set 
$B=\Aff(K^2)\cap \Elem(K^2)$. Equivalently, 
$B$ is the group of affine transformations of the form

\centerline{$(x,y)\in K^2\mapsto (z_1 x +u, z_2y +cx +v)$}

\noindent for some $z_1$ $z_2\in K^*$ and $c,\,u,\,v\in K$.
Set $B_0=B\cap GL(K^2)$
Indeed $B$ and $B_0$ are Borel subgroups of 
$\Aff(K^2)$ and $GL(K^2)$.
Moreover we have

\centerline{$\Aff(K^2)/B=GL(K^2)/B_0\simeq \P^1_K$.}

\bigskip
\noindent
{\it 1.7 The van der Kulk Theorem}
 
\noindent Recall the classical  

\begin{vdKthm}\cite{vdK} We have 

\centerline{$\Aut\,K^2\simeq \Aff(K^2)*_{B}\Elem(K^2)$, and}

\centerline{$\Aut_0\,K^2\simeq \GL(K^2)*_{B_0}\Elem_0(K^2)$.}

\end{vdKthm}

\section{Amalgamated Products and
Linearity.}

\noindent In this section, it is shown that, under a mild assumption, an amalgamated product $G_1*_A G_2$
which is linear over a ring is also linear over a field,
see Lemma \ref{criterion}.

\bigskip 
\noindent {\it 2.1 Linearity Properties}

\noindent
Let $R$ be a ring. For any $n\geq n$, let
$GL(n,R)$ be the group of invertible $n\times n$ matrices
with entries in $R$. Indeed for a field $K$,
the notation $GL(K^2)$ denotes the subgroup of $\Aut\,K^2$ of linear automorphisms of $K^2$ although 
$GL(2,K)$ is the abstract linear group of degree $2$.

A group $\Gamma$ is called 
{\it linear over $R$} if there is a group embbeding
$\Gamma\subset GL(n, R)$ for some integer $n$. The group
$\Gamma$ is called {\it linear} (or {\it linear over a ring}) if it is linear over some ring $R$. Similarly $\Gamma$ is called {\it linear over a field} if it is linear over some field $K$.

For a group $\Gamma$,  the strongest form of  linearity is the linearity over a field. On the opposite, there are also   groups 
$\Gamma$ which contain a FG  subgroup which is
not linear, even over a ring: these groups are nonlinear in the strongest sense.

\bigskip\noindent
{\it 2.2 Minimal embeddings}

\noindent Let $R$ be a commutative ring, let $n\geq 1$ be an integer and let
$\Gamma$ be a subgroup of $GL(n,R)$. For any ideal $J$ of $R$, let
$GL(n,J)$ be the kernel of the natural map
$GL(n,R)\to GL(n,R/J)$. Any element $g\in GL(n,J)$ can be written as $g=\id+A$, where $A$ is a $n$-by-$n$ matrix with entries in 
$J$.

The embbeding $\Gamma\subset GL(n,R)$ is called {\it minimal} if
for any ideal $J\neq 0$ we have $\Gamma\cap GL(n,J)\neq \{1\}$.

\begin{lemma} \label{minimal}Let $\Gamma\subset GL(n,R)$. There exist an ideal $J$ with
$\Gamma\cap GL(n,J)=\{1\}$ such that the induced embbeding
$\Gamma\to\GL(n,R/J)$ is minimal.
\end{lemma}

\begin{proof} Since $R$ is not necessarily noetherian,
the  proof requires Zorn Lemma.

Let $\cal S$ be the set of all ideals $J$
of $R$ such that $\Gamma\cap GL(n,J)=\{1\}$. The set
$\cal S$ is partially ordered  with respect to the inclusion.  For any chain $\cal C\subset\cal S$, the ideal 
$\cup_{I\in{\cal C}}\,I$ belongs to $\cal S$. Therefore Zorn Lemma implies that $\cal S$ contains a maximal element $J$. It follows that the induced embbeding $\Gamma\to\GL(n,R/J)$ is minimal.
\end{proof}

\bigskip\noindent
{\it 2.3 Groups with trivial commutants}

\noindent
By definition, a group $\Gamma$ has
{trivial commutants} if the commutant of any nontrivial normal subgroup $K$ of $\Gamma$ is trivial. 
Equivalently, if $K_1$ and $K_2$ is are commuting invariant 
subgroups of $\Gamma$, then one of them is trivial.

\begin{lemma}\label{linearity} Let  $\Gamma$ be a  group with trivial commutants.

If $\Gamma$ is linear over a ring, 
then $\Gamma$ is also linear over a field.
\end{lemma} 

\begin{proof} By hypothesis we have
$\Gamma\subset GL(n,R)$ for some commutative ring $R$.
By lemma \ref{minimal}, it can be assumed that the embbeding 
$\Gamma\to GL(n,R)$ is minimal. 

Let $I_1,\,I_2$ be ideals 
of $R$ with $I_1.I_2=0$. Let $A_1$ (respectively $A_2$) be
an arbitrary $n$-by-$n$ matrix with entries in $I_1$ 
(respectively in $I_2$) and set $g_1=1+A_1$, 
$g_2=1+A_2$. Since we have $A_1.A_2=A_2.A_1=0$, 
we have $g_1.g_2=g_2.g_1$, therefore 
$GL(n,I_1)$ and $GL(n,I_2)$ are commuting invariant subgroups
of $GL(n,R)$.

Since $K_1=\Gamma\cap GL(n,I_1)$ and  $K_2=\Gamma\cap GL(n,I_2)$
are commuting  invariant subgroups of $\Gamma$,  one of them is trivial. By minimality hypothesis, $I_1$ or $I_2$ is trivial.
Thus $R$ is prime.

It follows that $\Gamma\subset GL(n,K)$, where 
$K$ is the fraction field of $R$.
\end{proof}

\bigskip\noindent
{\it 2.4 The hypothesis ${\cal H}$}

\noindent
Let  $G_1$, $G_2$ be two groups sharing a common subgroup $A$ and
 set $\Gamma=G_1*_A G_2$.

 Let $S$ be the set of all finite sequences
 ${\bf i}=i_1,\dots,i_n$ of alternating $1$ and $2$.
For $i,\,j\in\{1,2\}$, let $S_{i,j}$ be
be the subset of all $\epsilon=(\epsilon_1,\dots,\epsilon_n)\in S$
starting with $i$ and ending with $j$. 
Let $\Gamma_{i,j}$ be the set of all $\gamma\in \Gamma$
of type $\epsilon$ for some $\epsilon\in S_{i,j}$.
Therefore we have
 
\centerline{ $\Gamma=A\sqcup \Gamma_{1,1}
 \sqcup \Gamma_{2,2}\sqcup \Gamma_{1,2}\sqcup \Gamma_{2,1}$.}

 By definition, the amalgamated product 
 $G_1*_A\,G_2$  is called {\it trivial} if
 $G_1=A$, or $G_2=A$ or if both $G_1$ and $G_2$ are isomorphic to
 $\Z/2Z$. The latter case is the uninteresting dihedral group
 $\Z/2Z*\Z/2Z$. 
 
 \noindent Let consider the following  hypothesis
 
 $({\cal H})$\hskip3mm For any $a\in A$ with $a\neq 1$, there is 
 $\gamma\in \Gamma$ such that
 $a^\gamma\notin A$
 
 \noindent where, as usual, $a^\gamma:=\gamma a \gamma^{-1}$.

 \begin{lemma}\label{conj} Let $\Gamma=G_1*_A\,G_2$ be a 
 nontrivial
 amalgamated product satisfying the hypothesis ${\cal H}$. Let $g\in\Gamma$ with $g\neq 1$.
 
 There are 
 $\gamma_1,\,\gamma_2\in \Gamma$ such that
 
 \centerline {$g^{\gamma_1}\in \Gamma_{1,1}$ and
 $g^{\gamma_2}\in \Gamma_{2,2}$}.
 
 In particular $\Gamma$ has trivial commutants.
 \end{lemma}

 \begin{proof}

  First it should be noted that $A$ cannot be 
 simultaneously a subgroup of index 2 in $G_1$ and in $G_2$.
 Otherwise, $A$ would be an invariant subgroup of $G_1$ and $G_2$,
 the hypothesis  ${\cal H}$ would imply that $A=\{1\}$
 and $\Gamma$ would be the trivial amalgamated product
 $\Z/2Z*\Z/2Z$. Hence it can be assumed that $G_2/A$ contains
 at least $3$ elements. 
 
 Now we prove the first claim, namely that any nontrivial
 conjugacy class intersects $\Gamma_{1,1}$ and $\Gamma_{2,2}$.

 It is clear that
 $G^*_i.\Gamma_{j,k}\subset \Gamma_{i,k}$
  whenever $i\neq j$. Similarly, we have
 $\Gamma_{i,j}.G^*_k\subset \Gamma_{i,k}$
 whenever $j\neq k$. 
 
 By hypothesis, $G^*_1$ is not empty. Let 
 $\gamma_1\in G^*_1$. We have 
 $\Gamma_{2,2}^{\gamma_1}\subset\Gamma_{1,1}$. Simlarly we have
 $\Gamma_{1,1}^{\gamma_2}\subset\Gamma_{2,2}$ for some
 $\gamma_2\in G_2^*$. Therefore the claim is proved for
 $g\in \Gamma_{1,1}\cup \Gamma_{2,2}$.
 
It is now enought to prove that for any 
 $g\in A\sqcup \Gamma_{1,2}\sqcup\Gamma_{2,1}$ with $g\neq 1$,
 there is $\gamma\in \Gamma$ such that $g^\gamma$ belongs to
 $\Gamma_{1,1}\cup \Gamma_{2,2}$. 
 
 Assume now $g\in \Gamma_{2,1}$. We have $g=u.v$ for some
 $u\in G^*_2$ and $v\in\Gamma_{1,1}$. Since 
 $[G_2:A]\geq 3$, there is $\gamma\in G^*_2$ such that
 $\gamma.u$ belongs to $G^*_2$. It follows that
 $\gamma.g$ belongs to $\Gamma_{2,1}$, and therefore
 $g^\gamma$ belongs to $\Gamma_{2,2}$. 
 
 Next, for $g\in \Gamma_{1,2}$, it inverse
 $g^{-1}$ belongs to $\Gamma_{1,2}$ and
  the claim follows from the previous point.
 
 Last, for $g\in A$,  there is $\gamma\in \Gamma$ such that $g^\gamma$ is not in $A$, by hypothesis ${\cal H}$. Thus $g^\gamma$ belongs to $\Gamma_{i,j}$  for some
 $i,\,j$. So $g$ is conjugate to some element in
 $\Gamma_{1,1}\cup\Gamma_{2,2}$ by the previous considerations.
 
 Next we prove that $\Gamma$ has trivial commutants.
 Let $K_1,\,K_2$ be nontrivial invariant subgroups. By the previous point, there are elements $g_1,\,g_2$ with
 
 \centerline{$g_1\in K_1\cap \Gamma_{1,1}$ and
 $g_2\in K_2\cap \Gamma_{2,2}$.}
 
Since we have $g_1 g_2\in \Gamma_{1,2}$ and $g_2 g_1\in \Gamma_{2,1}$, it follows that $g_1g_2\neq g_2g_1$. Therefore
$K_1$ and $K_2$ do not commute.
 
\end{proof}

 As an obvious corollary of Lemmas \ref{linearity} and
\ref{conj}, we get

\begin{lemma}\label{criterion} Let $\Gamma=G_1*_A\,G_2$ be a 
nontrivial amalgamated product satisfying the hypothesis ${\cal H}$.

If $\Gamma$ is linear over a ring, then it is linear over a field.
\end{lemma}

\section{A nonlinear FG subgroup of $\Aut_0\,\Q^2$.}

First we define a certain FG group $\Gamma=G_1*_A G_2$ and we will show that $\Gamma$ is nonlinear, even over a ring. Then we see that $\Gamma$ is
a subgroup of $\Aut_0\,\Q^2$ and therefore $\Aut_0\,K^2$ contains
many nonlinear FG subgroups for any characteristic zero field 
$K$, what proves Theorem A.2.

To show that $\Aut\,\Q^2$ is not linear over a field,
Cornulier uses a nonFG subgroup $G_{Cor}\subset \Aut\,\Q^2$ which is locally nilpotent 
but not nilpotent and therefore not linear over a field
\cite{C}. Nevertheless, the group
$G_{Cor}$ is linear over a ring, as it is shown at the end of the section.

\bigskip
\noindent {\it 3.1 Quasi-unipotent endomorphisms}

\noindent Let $V$ be a finite dimensional vector space over an algebraically closed field $K$. An element $u\in GL(V)$ is called
{\it quasi-idempotent} if all its eigenvalues are roots of one.
The {\it quasi-order} of a quasi-idempotent endomorphism
$u$ is the smallest positiver integer $n$ such that $u^n$ is unipotent.

\begin{lemma}\label{unipotent} Let $u\in GL(V)$. If $u$ and $u^2$ are conjugate,
then $u$ is quasi-unipotent and its quasi-order is odd.
If moreover $u$ has infinite order, then we have $\ch\,K=0$.
\end{lemma}

\begin{proof} Let $\Spec\,u$ be the spectrum of $u$.
By hypothesis the square map 
$\Spec\,u\to \Spec\,u, \lambda\mapsto\lambda^2$ is bijective,
hence for any $\lambda\in\Spec\,u$, we have 
$\lambda^{2^n}=\lambda$ for some integer $n$. It follows that all 
eigenvalues are odd roots of unity, what proves that $u$ is
quasi-unipotent of odd quasi-order.

Over any field of finite characteristic, the unipotent
endomorphisms have finite order. Hence we have $\ch\,K=0$ if
$u$ has infinite order.

\end{proof}

\noindent {\it 3.2 The Group $\Gamma=G_1*_A G_2$}

\noindent Set $G_1=\Z^2$, and let $\sigma, \sigma'$ be a basis $\Z^2$.
Set $\Z_{(2)}=\{x\in\Q\vert 2^n x\in\Z$ for $n>>0\}$.
Indeed $Z_{(2)}$ is the localization of the ring $\Z$ at $2$, but,
in what follows, we will only consider its group structure. The element $1$ in $Z_{(2)}$ will be denoted by $\tau$
and the addition in $Z_{(2)}$ will be  denoted multiplicatively.

Let $\sigma$ act on $\Z_{(2)}$ by multiplication by $2$, so we can consider the semi-direct product $G_2=\Z.\sigma\ltimes \Z_{(2)}$.
Let $\Gamma=G_1*_A G_2$, where $A=\Z\sigma$. It is easy to show
that $\Gamma$ is generated by $\sigma,\sigma'$ and $\tau$, and 
it is defined by the following two relations

\centerline{$\sigma\sigma'=\sigma'\sigma$, and
$\sigma \tau \sigma^{-1}=\tau^2$.}

\begin{lemma}\label{prepa1} The group
$\Gamma$ has trivial commutants.
\end{lemma} 

\begin{proof} First it is proved  that
the amalgamated product
$\Gamma=G_1*_A G_2$ satisfy hypothesis ${\cal H}$.

Indeed it is enought to show that,  that for any nontrivial
element $a\in A$, its conjugate $\tau^{-1} a \tau$ is not in $A$.
We have $a=\sigma^n$ for some $n\neq 0$, and it can be assumed that $n\geq 1$.

We have $\sigma^n \tau\sigma^{-n}=\tau^{2^n}$ and therefore
$\tau^{-1} a \tau=\tau^{2^n-1} a$ what proves that
$\tau^{-1} a \tau$ is not in $A$. 

Hence the amalgamated product
$\Gamma=G_1*_A G_2$ satisfy hypothesis ${\cal H}$. Therefore
by Lemma \ref{criterion}, $\Gamma$ has trivial commutants.
\end{proof}

\noindent{\it 3.3 Nonlinearity of $\Gamma$}

\begin{lemma} \label{nonlinear}

The group $\Gamma$ is not linear, even over a ring.

\end{lemma}

\begin{proof} 
Assume that $\Gamma$ is linear over a ring.
By Lemma \ref{prepa1} the group $\Gamma$
has trivial commutants. Thus, by Lemma \ref{linearity},
$\Gamma$ is linear over a field.

Let $\rho:\Gamma\to GL(V)$ be an embbeding, where $V$ is a finite dimensional vector space over a field $K$.
Since $\sigma \tau \sigma^{-1}=\tau^2$ it follows from 
Lemma \ref{unipotent} that $\rho(\tau)$ is quasi-unipotent and its
quasi-order $n$ is odd.  

Let $\Gamma'$ be the subgroup of $\Gamma$ generated by
$\sigma,\,\sigma'$ and $\tau^n$. Since the morphism
$\psi: \Gamma\to\Gamma'$ defined by $\psi(\sigma)=\sigma$,
$\psi(\sigma')=\sigma'$ and $\psi(\tau)=\tau^n$ is an isomorphism,
it can be can assumed that $\rho(\tau)$ is unipotent.

Since $\tau$ has infinite order,
$K$ has characteristic zero by Lemma \ref{unipotent}. Set
$u=\rho(\tau)$, $h=\rho(\sigma)$ and 

\centerline{$e=\log\,u=\log\,1-(1-u)=\sum_{n\geq 1} \,(1-u)^n/n$,}

\noindent which is defined since $1-u$ is nilpotent. Since
$h u h^{-1}=u^2$, we have 

\centerline{$h e h^{-1}=2e$.}

It can be assumed that $K$ is algebraically closed.
Let $V=\oplus_{\lambda\in K}\, V_{(\lambda)}$ be the decomposition of $V$ into generalized eigenspace relative to $h$. Also, for 
$\lambda\in K$, set
$V_{(\lambda)}^+=\oplus_{n\geq 0}V_{(2^n\lambda)}$. 

Let $\lambda\in K$. Since
$\rho(G_1)$ commutes with $h$ we have 

\centerline{$\rho(G_1)V_{(\lambda)}\subset V_{(\lambda)}$.}

\noindent  Moreover we have
$e.V_{(\lambda)}\subset V_{(2 \lambda)}$ and therefore

\centerline{$\rho(G_2)V_{(\lambda)}^+\subset V_{(\lambda)}^+$.}

\noindent It follows that $V_{(\lambda)}^+$ is a $\Gamma$-module.
Since $\tau$ acts trivially on $V_{(\lambda)}^+/V_{(2\lambda)}^+$,
it follows that the image of $\Gamma$ in $GL(V_{(\lambda)}^+/V_{(2\lambda)}^+)$ is commutative. It follows that $V$ has a composition series by one-dimensional $\Gamma$-module, hence
$\rho(\Gamma)$ is solvable.

Since it is solvable, $\Gamma$ contains a nontrivial invariant abelian subgroup.
This fact contradicts  Lemma \ref{prepa1}, which states
that $\Gamma$ has trivial commutants. 
\end{proof}

\noindent{\it 3.4 Proof of Theorem A.2}

\begin{MainA2} Let $K$ be a field of characteristic zero. Then
$\Aut_0\,K^2$ contains FG subgroups which are not linear even over a ring.
\end{MainA2}

\begin{proof} Let define three automorphisms
$S, S'$ and $T$ of $K^2$ as follows. First 
$S$ and $S'$ are linear automorphisms  where
$S=1/2\,\id$ and $S'$ is defined by the matrix
$\begin{pmatrix}
1 & 1 \\
1 & 0
\end{pmatrix}$.  Next $T$ is the quadratic automorphism
$(x,y)\mapsto (x, y+x^2)$. Set $K_1=<S,S'>$,
 $K_2=<S,T>$, $C=<S>$.

The eigenvalues of $S'$ are $1\pm\sqrt{5}\over 2$. Hence
$n\neq 0$ $S'^n$ is not upper diagonal. It follows that
 $K_1\cap B_0=C$. Moreover it is clear that 
 $K_2\cap B_0=C$. It follows from Lemma \ref{subamal} that
 $K_1*_C K_2$ is a subgroup of $\GL(K^2)*_{B_0} \Elem_0(K^2)$. 
 
 Since $S T S^{-1}=T^2$ is is clear that the
 group morphism $\Gamma\to K_1*_C K_2,
 \sigma\mapsto S, \sigma'\mapsto S', \tau\mapsto T$ is an isomorphism. Therefore $\Gamma$ is a subgroup of $\Aut_0\,K^2$.
 
 Hence, by Lemma \ref{nonlinear}  $\Aut_0\,K^2$ contains
 a FG subgroup which is not linear, even over a ring.

\end{proof}

\noindent{\it 3.5 About  Counulier's subgroup.}

\noindent The group $G_{Cor}$ considered in \cite{C} is
the subgroup of all automorphisms of $\Q^2$ of the form

\centerline{$(x,y)\mapsto (x+u, y+f(x))$.}

\noindent This group, which is not FG, has been used in \cite{C} to prove

\begin{Cornu} The group $G_{Cor}$ is not linear over a field.
Consequently, the group $\Aut\,\Q^2$ is not linear over a field.

\end{Cornu}

\noindent
Indeed, $G_{Cor}$ is an example of a group which is linear over 
a ring, but not linear over a field.
Set $R=\Q[[x]]\oplus \Q((x))/\Q[[x]]$, where $\Q((x))/\Q[[x]]$ is a  square-zero ideal.

\begin{proposition} The group $G_{Cor}$ is linear over the ring $R$.
\end{proposition}

\begin{proof} Let $T\simeq\Q$ be the group whose elements are denoted
$\tau^{\alpha}$ for $\alpha\in\Q$ and the product is given
$\tau^{\alpha}.\tau^{\beta}= \tau^{\alpha+\beta}$. As a group,
$G_{Cor}$ is isomorphic to $T\ltimes\Q[t]$, where the action of $T$ on
$\Q[t]$ is defined by

\centerline{$\tau^{\alpha} f(t)\tau^{-\alpha}=f(t+\alpha)$}

\noindent for any $\alpha\in\Q$ and $f\in \Q[t]$.

Recall that any $f(t)$ can be written as a finite sum
$f(t)=\sum_{n\geq 0}\,a_n (^t_n)$.
There are  embeddings 
$\rho_1:\Q[t]\to GL(2,R)$ and
$\rho_2:T\to GL(2,R)$
defined by

\centerline{
$\rho_1(\sum_{n\geq 0} \, a_n \, (^t_n))= \begin{pmatrix}
1 & \sum_n \,a_n [x^{-n-1}] \\
0& 1
\end{pmatrix}$ and  
$\rho_2(\tau^{\alpha})= \begin{pmatrix}
1 & 0 \\
0& (1+x)^\alpha
\end{pmatrix}$,}

\noindent where, for $n<0$ the symbol $[x^{-n}]$ is the element
$x^{-n}\mod \Q[[x]]$ in the ideal $\Q((x))/\Q[[x]]$ of $R$, and the 
fractional power $(1+x)^\alpha$ is the formal series $\sum_{k\geq 0}\,
(^\alpha_k)\,x^k$.

We claim that

\centerline{$\rho_2(\tau^{\alpha}) 
\rho_1(f(t))\rho_2(\tau^{-\alpha})=\rho_1(f(t+\alpha))$.}

\noindent Indeed for $\alpha=1$, this formula follows from the fact that $(^{t+1}_n)-(^t_n)=(^{\,\,t}_{n-1})$. It follows that 
the formula holds for any integer $\alpha\geq 1$. Since both sides of the identity are   polynomials in $\alpha$, the formula holds for any $\alpha\in\Q$. 

It follows that $\rho_1$ and $\rho_2$ can be combined into
an embbeding $\rho: G_{Cor}\to GL(2,R)$.
\end{proof}

\section{Semi-algebraic characters.}

Let $K$ be an infinite field of characteristic $p$. 

\begin{lemma}\label{span} Let $n\geq 1$ be an integer prime to $p$. Then for any $x\in K$, there is an integer $m$ and a collection $x_1,\,x_2\dots x_m$ of elements of $K^*$ such that

\centerline{$x=\sum_{1\leq k\leq m}\,x_i^n$.}

\end{lemma}

\begin{proof} Let $K'$ be the additive span of
the set $\{y^n\vert y\in K\}$. Since 
$-y^n=(p-1) y^n$, $K'$ is an additive subgroup of $K$. Clearly,
$K'$ is a subring of $K$. Let $x\in K'$ with $x\neq 0$. It follows from the formula $x^{-1}=(1/x)^n x^{n-1}$ that $K'$ is a subfield.

We claim that $K'$ is infinite. Assume otherwise. Then
$K'=\F_q$ for some power $q$ of $p$. Since any element of $K$ is
algebraic over $K'$ of degree $\leq n$, it follows that
$K\subset \cup_{k\leq n} \F_{q^k}$, a fact that contradicts that $K$ is infinite.

It remains to prove that $K'=K$. 
For $x\in K$, let $P(t)\in K[t]$ be the polynomial $P(t)=(x+t)^n$. Let $y_0,\dots y_n$ be $n+1$
distinct elements in $K'$. We have $P(y_i)\in K'$ for any
$0\leq i\leq n$. Using Lagrange's interpolation polynomials, there exist  a polynomial
$Q(t)\in K'[t]$ of degree $\leq n$ such that $P(y_i)=Q(y_i)$ for any $0\leq i\leq n$. Since $P-Q$ has at least $n+1$ roots, we have
$P=Q$ and therefore $P(t)\in K'[t]$. Since

\centerline{$P(t)=t^n + nx t^{n-1}+\dots$}

\noindent it follows that $nx$ belongs to $K'$, and therefore $x\in K'$.
Hence $K'=K$. 

\end{proof}

\begin{lemma}\label{n=m} Let $K$ be an infinite field of characteristic $p$,
let $\mu:K\to K$ be a field automorphism and
let  $n,\,m$ be positive integers which are prime to $p$. Assume that

\centerline {$x^n=\mu(x)^m$, for all $x\in K^*$.}

Then we have $m=n$.

\end{lemma}

\noindent
{\it Remark:} More precisely,  the hypotheses of the lemma
imply also that $\mu=\id$, but this is not required in what follows. 

\begin{proof} 

{\it Step 1: proof for $K=\F_p(t)$.}
There are $a,\,b,\,c$ and $d\in K$ 
with $ad-bc\neq 0$ such that
$\mu(t)=\frac{at+b}{ ct+d}$.  The identity

\centerline{$t^n=(\frac{at+b}{ct+d})^m$}

\noindent clearly implies that $n=m$.

\smallskip
\noindent {\it Step 2: proof for $K\subset \overline\F_p$.}
Since $K$ is infinite, $K$ contains arbitrarily big finite fields. So we have
$K\supset \F_{p^N}$ for some positive integer $N$  with $nm<p^N$. Since $K$ contains a unique
field of cardinality $p^N$ we have $\mu(\F_{p^N})=\F_{p^N}$.
There is an non negative integer $a<N$ such that
$\mu(x)= x^{p^a}$ for any $x\in \F_{p^N}$. Therefore we have

\centerline{$x^n=x^{m{p^a}}$ for all $x\in F_{p^N}^*$.}

\noindent It follows that $n\equiv m{p^a}\, \mod p^N-1$. 
Let $n=\sum_{k\geq 0}\, n_k p^k$ and
$m=\sum_{k\geq 0}\, m_k p^k$ be the $p$-adic expansions of $n$ and $m$. By definition each  digit $n_k$, $m_k$ is an integer between $0$ and $p-1$. Since $n.m<p^N$, we have 
$m_k=m_k=0$ for $k\geq N$.

For each integer $k$ let $[k]$ be its residue modulo $N$, so
we have $0\leq [k]<N$ by definition. We have

\centerline{$n\equiv mp^a\equiv \sum_{0\leq k<N} m_k p^{[a+k]}\, \mod p^N-1$.}

\noindent 
Since both integers $n$ and $\sum_{0\leq k<N} m_k p^{[a+k]}$
blongs to $[1,p^N-1]$ and are congruent modulo $p^N-1$, it
follows that

\centerline{$n= \sum_{0\leq k<N} m_k p^{[a+k]}$.}

 Assume $a>0$. 
We have $n_a=m_0$ and $n_0=m_{N-a}$.
Since $n$ and $m$ are prime to $p$, the digits $n_0$ and
$m_0$ are not zero, therefore we have
$n_a\neq 0$ and $m_{N-a}\neq 0$. It follows that

\centerline{$n\geq p^a$ and $m\geq p^{N-a}$,}

\noindent which contradicts  that $nm<p^N$.

Therefore, we have $a=0$, and the equality $n=m$ is obvious.

\smallskip
\noindent{\it Step 3: proof for any infinite $K$.} 
In view of the previous step, it can be assumed that
$K$ contains a transcendental element $t$. Therefore,
$K$ contains the subfield $L:=\F_p(t)$.

By Lemma \ref{span}, $L$ and $\mu(L)$ are the additive span of the set 
$\{x^n\vert x\in K^*\}=\{\mu(x)^m\vert x\in K^*\}$. Therefore,
we have $\mu(L)=L$. Since $\mu$ induces a field automorphism of $L$,
the equality $n=m$ follows from the first step.

\end{proof}

Let $L$ be another field. A group morphism 
$\chi:K^*\to L^*$ is called a {\it character} of $K^*$.
Let $X(K^*)$ be the set of all $L$-valued
characters of $K^*$. 
Let $n$ be an integer prime to $p$. A character
$\chi\in X(K^*)$ is called {\it semi-algebraic} of {\it degree}
$n$ if 

\centerline{$\chi(x)=\mu(x)^n$, for any $x\in K^*$,}

\noindent for some  field embedding $\mu: K\to L$.

Let ${\cal X}_n(K^*)$  be the set of all semi-algebraic characters of $K^*$ of degree $n$. Of course it can be assumed that $L$ has characteristic $p$, otherwise the set ${\cal X}_n(K^*)$ is empty. 

The next lemma shows that for an integer $n>0$ prime to $p$, the degree 
is uniquely determined by the character $\chi$. Indeed this statement is also true for the
negative integers $n$ and moreover the  field embbeding $\mu$ is also determined by $\chi$. For simplicity of the exposition, the lemma is stated and proved in its minimal form. 

However, it should be noted that the condition that $n$ is prime to $p$ is essential. 
Indeed, if $\mu:K\to L$ is a field embbeding into a perfect field $L$, then
$\mu(x)^n=\mu'(x)^{pn}$, where $\mu'$ is the field embedding
defined by $\mu'(x)=\mu(x)^{1/p}$.

\begin{lemma}\label{degree} Let $n\neq m$ be positive integers which are prime to $p$. Then ${\cal X}_n(K^*)\cap {\cal X}_m(K^*)=\emptyset$.
\end{lemma}

\begin{proof} Let $n$, $m$ be positive integers prime to $p$
and let $\chi\in {\cal X}_n(K^*)\cap {\cal X}_m(K^*)$.
By definition, there are fields embeddings 
$\nu, \nu':K\to L$ such that

\centerline{$\chi(x)=\nu(x)^n=\nu'(x)^m$ for any $x\in K^*$.}

\noindent By Lemma \ref{span}, $\nu(K)$ and $\nu'(K)$ are the linear span of $\Image\,\chi$. Therefore
we have $\nu(K)=\nu'(K)$ and 
$\mu:=\nu'^{-1}\circ\nu$ is a well defined field automorphism of 
$K$. We have

\centerline{
$\mu(x)^n=\nu'^{-1}\circ\nu(x^n)= \nu'^{-1}\circ\nu'(x^m)=x^m$.}

Therefore by Lemma \ref{n=m}, we have $n=m$.

\end{proof}

\section{Nonlinearity of  $\Aut_0\,K^2$ for $K$ infinite of characteristic $p$.}

This section provides the proof of Theorem A.1. Since Theorem A.2 has been proved in 
Section 3, it can be assumed that $K$ is a field of characteristic $p$.
Although our setting is different from  \cite{BT}, this section 
follows the same idea, namely that some abstract morphisms of algebraic groups are, somehow, semi-algebraic.

Recall that an {\it elementary abelian $p$-group} is simply a
$\F_p$-vector space $E$ viewed as a group. Its $\F_p$-dimension
is called the {\it rank} of $E$.

\begin{lemma}\label{inf-rank} Let $L$ be a field and let $E$ be
an elementary abelian $p$-group of infinite rank.
If $\ch\,L\neq p$, then $E$ is not linear over $L$.
\end{lemma}

\begin{proof} It can be assumed that $L$ is algebraically closed.
Let $V$ be a  vector space over $L$ of dimension $n$ and let
$F\subset GL(V)$ be an elementary abelian $p$-subgroup. For any character $\chi: K\to L^*$, let $L_\chi$ be the 
corresponding one-dimensional representation of $F$. 

Since $\ch\L\neq p$, $F$ is a diagonalizable subgroup of $GL(V)$.
So there is an isomorphism of $F$-modules 
 
 \centerline{$V\simeq \oplus_{1\leq i\leq n}\,L\,_{\chi_i}$.}
 
\noindent Since $\cap \,\Ker\,\chi_i$ has $\F_p$-codimension
 $\leq n$, it follows that the rank of $F$ is $\leq n$.

 Therefore no infinite rank elementary abelian $p$-group
 is  linear over $L$.

\end{proof}

From now on, $L$ is an algebraically closed field of characteric 
$p$.

\begin{lemma}\label{reco} Let $\chi\in X(K^*)$, let $\mu:K\to L$
be a nonzero additive map and let $n$ be a positive
integer prime to $p$.

Assume that

\centerline{$\mu (x^ny)=\chi(x)\mu(y)$}

\noindent for any $x\in K^*$ and $y\in K$. 

Then
$\chi$ is a semi-algebraic character of degree $n$.

\end{lemma}

\begin{proof} By Lemma \ref{span}, $K$ is the additive span
of $(K^*)^n$. Therefore there is some $x\in K^*$ such that
$\mu(x^n)\neq 0$. Since $\mu(x^n)=\chi(x)\mu(1)$, it follows that
$\mu(1)\neq 0$. After rescaling $\mu$, one can assume that
$\mu(1)=1$.

Let $x,\,y \in K^*$. We have $\mu(x^n)=\chi(x)\mu(1)=\chi(x)$,
and therefore $\mu(x^n y^n)=\chi(x)\mu(y^n)=\mu(x^n)\mu(y^n)$.
By lemma \ref{span},  $K$ is the additive span
of $(K^*)^n$. It follows that

\centerline{$\mu(ab)=\mu(a)\mu(b)$}

\noindent for any $a,\,b\in K$. Hence $\mu$ is a field embedding.
Moreover, we have $\chi(x)=\mu(x^n)$ for any $x\in K^*$. Therefore
$\chi$ is a  semi-algebraic character of degree $n$.
\end{proof}

For $n\geq 1$, let $G_n(K)$ be the semi-direct product 
$K^*\ltimes K$, where any $z\in K^*$ acts on $K$ as $z^n$.
More explicitely, the elements of $G_n(K)$ are denoted 
$(z,a)$, with $z\in K^*$ and $a\in K$ and the product
is defined by
   
 \centerline{  $(z,a).(z',a')=(zz',z'^na+a')$,}
 
 \noindent for any $z,\,z'\in K^*$ and $a,\,a'\in K$.

Let $V$ be a finite dimensional $L$-vector space and
let $\rho:G_n(K)\to GL(V)$ be a group morphism. With respect to 
the subgroup  $(K^*,0)$ of $G_n(K)$, there is a 
decomposition of $V$ as 

\centerline{$V=\oplus_{\chi\in X(K^*)}\, V_{(\chi)}$}

\noindent where 
$V_{(\chi)}:=\{v\in V\vert \forall x\in K^* :(\rho(x,0)-\chi(x))^n v=0$
for $n>>0\}$ is the generalized eigenspace associated with the character $\chi$.
Here and in the sequel, $\rho(z,a)$ stands for $\rho((z,a))$, 
for any $(z,a)\in G_n(K)$.

\begin{lemma}\label{char} Let $n$ be a positive integer prime to $p$.
$\rho:G_n(K)\to GL(V)$ be an injective morphism. Then we have

\centerline{$\End (V)_{(\chi)}\neq 0$}

\noindent for some $\chi\in {\cal X}_n(K^*)$.

\end{lemma}

\begin{proof}
Set

\centerline{$V_0=\{v\in V\vert \rho(1,a)v=v,\, \forall a \in K\}$, and}

\centerline{
$V_1=\{v\in V\vert \rho(1,a)v=v \,\mod V_0 ,\,\forall a \in K\}$.}

\noindent Since  $\rho(1,a)^p=0$ for any $a\in K$, the $\rho(K)$-module
$V$ is unipotent.  It follows that

\centerline{$V_1\supsetneq V_0\neq 0$.}

Clearly, $V_1$ is a 
$G_n(K)$ submodule, and let $\rho_1:G_n(K)\to GL(V_1)$
be the restriction of $\rho$ to $V_1$.
Let $\theta: K\rightarrow \End(V_1)$ be the map defined by
 $\theta(a)=\rho_1(1,a)-1$ for $a\in K$.
 By definition, we have  $\theta(a)(V_1)\subset V_0$ and
 $\theta(a)(V_0)=\{0\}$. Hence we have 
  $\theta(a)\circ\theta(b)=0$ for any $a,\,b\in K$. Since

\centerline{$\rho_1(1,a+b)=1+\theta(a+b)$, and}

\centerline{$\rho_1(1,a)\circ \rho_1(1,b)=(1+\theta(a))(1+\theta(b))
=1+\theta(a)+\theta(b)$,}

\noindent it follows that $\theta(a+b)=\theta(a)+\theta(b)$.

Let $W$ be the $L$-vector space generated by $\Image\, \theta$. Since
$\theta\neq 0$, we have  $W\neq 0$. Let $\rho_W$ the action by conjugacy 
of $G_n(K)$ over $W$. We have $\rho_W(1,K)=0$, therefore there is a $G_n(K)$-equivariant map $g:W\to W'$, where
$W'$ is a one dimensional quotient of $W$.
Therefore there is some $\chi\in X(K^*)$ such that 

\centerline{$\rho_{W'}(z,0)=\chi(z)$,}

\noindent for any $z\in K^*$, where $\rho_{W'}$ is the induced action on $W'$.

Set $\mu=g\circ \theta$. It follows from the previous computation
that $\mu$ is additive. Moreover if follows from the identity
$\rho(x,0)\rho(1,a)\rho(x^{-1},0)=\rho(1,x^n a)$ that

\centerline{$\rho_1(x,0)\theta(a)\rho_1(x^{-1},0)=\theta(x^n a)$}

\noindent  and therefore

\centerline{$\mu(x^n a)=\chi(x) \mu(a)$,}

\noindent for any $x\in K^*$ and $a\in K$. 
By Lemma \ref{reco}, the character $\chi$ is algebraic of degree $n$. Since $W'$ is a subquotient of $\End(V)$, it follows that
$\End(V)_{(\chi)}\neq 0$.

\end{proof}

Let $n\geq 1$. From now on, we will identify
$G_n(K)$ with the subgroup of  $\Aut\,K^2$ of all
automorphism of the form

\centerline{$(x,y)\mapsto (zx, zy +a x^{n+1})$}

\noindent for some $z\in K^*$ and $a\in K$.

 \begin{lemma}\label{NLp} Let $K$ be an infinite field of characteritic $p$ and let $E$ be a subgroup of
 $\Elem_0(K^2)$. 
 
 If $E$ contains 
 $G_n(K)$ for infinitely many integers $n$ prime to
 $p$, then $E$ is not linear over a field.
 In particular, $\Elem_0(K^2)$ is not linear over a field.
 \end{lemma}

 \begin{proof} Assume otherwise, and let
 $\rho: E\to GL(V)$, where $V$ is a finite dimensional vector space over an algebraicall closed field $L$. The group
 $E$ contains some elementary abelian $p$-groups of
 infinite rank. Therefore, by Lemma \ref{inf-rank}, the field
 $L$ has characteristic $p$.
 
 The group $K^*$ is identified with the subgroup of
 $\Aut\,K^2$ of linear homotheties. Note that 
 $K^*$ is a common torus of all subgroups $G_n(K)$ for any $n\geq 1$.
 Let $\Omega$ be the set of all characters $\chi$ of
 $K^*$, such that $\End(V)_{(\chi)}\neq 0$
 
By Lemma \ref{char}, $\Omega$ contains a
 character in $\chi_n\in {\cal X}_n(K^*)$ for infinitely many positive integers
 $n$ prime to $p$. By Lemma \ref{degree} these characters are all
 distincts, which contradicts that $\Omega$ is a finite set.
 
 \end{proof}
 
 \begin{lemma}\label{H} The amalgamated product
 $\GL(K^2)*_{B_0}\,\Elem_0(K^2)$ satisfies hypothesis
 ${\cal H}$.
 \end{lemma}

 \begin{proof} Let $g\in B_0$ with $g\neq 1$.
 
 First, if $g$ is not an homothety,
 there is $\gamma\in GL(K^2)$ such that $g^\gamma$ is not upper triangular and therefore we have $g^\gamma\notin B_0$. 
 
 Otherwise $g$ is an homothety with ratio $\lambda\neq 1$.
Let $\gamma$ be the automorphism $(x,y)\mapsto (x,y+x^2)$. 
Then  $g^\gamma$ is the automorphism

\centerline{$(x,y)\mapsto (\lambda x,\lambda y +(\lambda^2-\lambda) x^2)$,}

\noindent so $g^\gamma$  is not in $B_0$. 
 
\end{proof}

 \begin{MainA2} If $K$ be an infinite field of characteritic $p$,
then $\Aut_0\,K^2$ is not linear even over a ring.
\end{MainA2}

\begin{proof}  Assume otherwise. Using van der Kulk Theorem
and lemmas \ref{H}, \ref{criterion} and \ref{linearity}
it follows that $\Aut_0\,K^2$ is linear over a field. 
This contradicts the fact that 
its subgroup $\Elem_0(K^2)$ is not linear over a field by
Lemma \ref{NLp}.
\end{proof}

 \section{The Tits Ping-Pong}

\bigskip
\noindent
{\it 6.1 The Ping-Pong lemma}

\noindent
Let $(E_p)_{p\in P}$ be a collection of groups indexed by
a  set $P$. Let $\Gamma:=*_{p\in P}\,E_p$ be the free product of
these groups. Let $S$ be the set of all finite sequences
$(p_1,\dots,p_n)$ such that $p_i\neq p_{i+1}$ for any
$i<n$. For each $p\in P$, set 
$E^*_p=E\setminus \{1\}$.

Let $\gamma\in \Gamma$. There is a unique
${\bf p}=(p_1,\dots,p_n)\in P$ and a unique decomposition  
of $\gamma$

\centerline{$\gamma=\gamma_1\dots\gamma_n$,} 

\noindent
where $\gamma_i\in E^*_{p_i}$. The sequence ${\bf p}$ is
called the {\it type} of $\gamma$.

The free product $\Gamma:=*_{p\in P}\,E_p$ is called {\it nontrivial} if 

(i) for any $p\in P$, $E_p\neq\,\{1\}$,

(ii) $\Card\,P\geq 2$, and

(iii) if $\Card\,P= 2$, $\Gamma$ is not the trivial free product
$\Z/2\Z*\Z/2\Z$.

\begin{lemma}\label{Ping-Pong} Assume that a nontrivial 
free product $\Gamma=*_{p\in P}\,E$ acts on some set $\Omega$. Let $(\Omega_p)_{p\in P}$ be a collection of subsets 
in $\Omega$. Assume

(i) The subsets $\Omega_p$ are nonempty and disjoint, and

(ii) we have $E^*_p.\Omega_q\subset \Omega_p$ whenever $p\neq q$.

Then the action of $\Gamma$ on $\Omega$ is faithful.
\end{lemma}

\noindent The hypothesis (ii) is called a Ping-Pong hypothesis.

\begin{proof} Let $\gamma\in\Gamma$ with $\gamma\neq 1$,
and let ${\bf p}=(p_1,\dots, p_n)$ be its type. 
We claim that there is $x\in \Omega$ such that $\gamma.x\neq x$

If $\Card\,P\geq 3$ let $q\in P$ with 
$q\neq p_1$ and $q\neq p_n$. Let $x\in \Omega_q$.
By the Ping-Pong hypothesis,  $\gamma.x$ belongs to
$\Omega_{p_1}$, therefore $\gamma.x\neq x$. 

If $\Card\,P=2$ it can be assumed that $P=\{1,2\}$. It follows from Lemma \ref{conj} that $\gamma$ has a conjugate $\gamma'$ of
type ${\bf q}=(q_1,\dots,q_m)$ with $q_1=q_m=1$. Let 
$x'\in \Omega_2$. By the Ping-Pong hypothesis,  $\gamma'.x'$ belongs to
$\Omega_{1}$, therefore $\gamma'.x\neq x'$. It follows that there is some $x\in\Omega$ with $\gamma.x\neq x$.

In both cases, any $\gamma\neq 1$ acts nontrivially.
Hence  the action of $\Gamma$ on $\Omega$ is faithful.
 \end{proof}

\noindent {\it 6.2 Mixture of free products, amalgamated products and semi-direct product}

\noindent This section is devoted to three technical lemmas,
for groups with a mix of free products, amalgamated products and
semi-direct products.

Let $G$ be a group. A {\it $G$-structure} on a group $E$ is a $G$-action on $E$, 
where $G$ acts by  group automorphisms. Equivalently, it means that the semi-direct product 
$G\ltimes E$ is well defined. For simplicity, a group $E$ with a $G$-structure is called a 
{\it $G$-group}. Two $G$-groups $E$, $E'$  are called
{\it $G$-isomorphic} if there is an isomorphism from $E$ to $E'$ which commutes with the
$G$-structure.

Let $P$ be a set on which $G$ acts, let $E$ be a group and 
for each $p\in P$ let $E_p$ be a copy of $E$. A $G$-structure on $*_{p\in P}\, E_p$ is called
{\it compatible with the $G$-action on $P$} if $E_p^g=E_{g(p)}$ for any $g\in G$ and $p\in P$. Roughly speaking, 
it means that $G$ acts on  $*_{p\in P}\, E_p$ by permuting the free factors. When it is not ambiguous, we just say that the 
$G$-structure is {\it compatible}. 

\noindent 

Now let $G$, $U$ be two groups sharing a common group
$A$. Assume moreover that $U=A\ltimes E$,
for some normal subgroup $E$ of $U$. Set
$\Gamma_0=G*_{A}U$. The natural map 
$U\to U/E\simeq A$ induces a group morphism

\centerline {$\chi: \Gamma_0=G*_{A}U\to G=G*_A A$.}

\noindent and let $\Gamma_1$ be its kernel.

Set $P=G/A$. For $\gamma\in G$, the subgroup
$E^\gamma$ lies obviousy in $\Gamma_1$ and it depends only on
$\gamma\mod A$.

\begin{lemma}\label{mixing1} The group $\Gamma_1$ is the free product
of all $E^\gamma$, where $\gamma$ runs over $P$.
\end{lemma} 

\begin{lemma}\label{mixing2} For each $p\in P$, let $E_p$ be
a copy of $E$. A compatible $G$-structure on $*_{p\in P} E_p$
obviously provides a $A$-structure on $E$, and we have

\centerline{$G\ltimes (*_{p\in P} E_p)\simeq 
G*_{A} (A\ltimes E)$.}
\end{lemma} 

\begin{proof} Set $\Gamma'_1=*_{\gamma\in P}\, E^\gamma$.
Clearly $G$ acts over $\Gamma'_1$, so we can consider

\centerline{$\Gamma'_0:=G\ltimes \Gamma'_1$.}

\noindent Using the universal properties of amalgamated, free and
semi-direct products, one defines morphisms
$\phi:\Gamma_0\to \Gamma'_0$ and $\psi:\Gamma'_0\to \Gamma_0$
which are inverse of each other. It follows that 
$\Gamma_0$ and $\Gamma'_0$ are isomorphic, and 
$\phi$ induces an isomorphism from $\Gamma_1$
to $\Gamma_1'$, which proves Lemma \ref{mixing1}.

Let $1$ be the distinguished point of
$G/A$. We have $E_1^a=E_1$ for any $a\in A$, hence the group
$E=E_1$ has an $A$-structure. The rest of the proof 
of Lemma \ref{mixing 2} follows from
universal properties, as before.

\end{proof}

For the last lemma, let $G$ be a group acting on a set $P$. Let
$E$ and $E'$ be two groups. For each $u\in P$, let 
$G_u$ be the stabilizer in $G$ of $u$ and 
$E_p$ (respectively $E'_p$) be a copy of $E$ (respectively of
$E'$). 

Assume given some compatible $G$-structures on 
$*_{p\in P}\,E_p$ and $*_{p\in P}\,E'_p$. Obviously, it
provides some $G_p$-structure on  $E_p$ and $E'_p$, for any $p\in P$.

\begin{lemma}\label{mixing3} Assume that the groups $E_p$ and $E'_p$ are $G_p$-isomorphic 
for any $p\in P$. Then the groups $G\ltimes(*_{p\in P}\,E_p)$ and 
$G\ltimes(  *_{p\in P}\,E'_p)$ are isomorphic.
\end{lemma}

\begin{proof} First, assume that $G$ acts transitively on $P$. Let $u$ be a point of $P$, 
and let $A:=G_u$ be its stabilizer. It follows from Lemma \ref{mixing2} that

\centerline{$G\ltimes(*_{p\in P}\,E_p)   \simeq G*_{G_u}(G_u\ltimes E_p))$, and}

\centerline{$G\ltimes(*_{p\in P}\,E'_p)\simeq G*_{G_u}(G_u\ltimes E'_p)$.}

\noindent Hence  the groups 
$*_{p\in P}\,E_p$ and $*_{p\in P}\,E'_p$ are $G$-isomorphic. 
From this,  it follows that the groups 
$*_{p\in P}\,E_p$ and $*_{p\in P}\,E'_p$ are 
$G$-isomorphic even if $G$ does not act transitively on $P$. Therefore $G\ltimes(*_{p\in P}\,E_p)$ and 
$G\ltimes(  *_{p\in P}\,E'_p)$ are isomorphic.
\end{proof}

\bigskip
\noindent
{\it 6.3 The subgroups of $GL_S(2,K[t])$ in $GL(2,K[t])$}

\noindent For $G(t)\in GL(2,K[t])$, let $G(0)$ its evaluation at $0$. For any subgroup $S\subset GL(2,K)$
set

\centerline
{$GL_S(2,K[t]):=\{G(t)\in GL(2,K[t])\vert G(0)\in S\}$.}

\noindent For $S=\{1\}$, the  group $GL_S(2,K[t])$ will be denoted by
$GL_1(2,K[t])$.

For any $\gamma\in\P^1_K$, let $e_\gamma\in\End(K^2)$ such that
$e_\gamma^2=0$ and $\Image e_\gamma=\gamma$. Since
$e_\gamma$ is unique up to a constant multiple, the group

\centerline{$E_\gamma:=\{ \id + tf(t) e_\gamma \vert\,f\in K[t]\}$}

\noindent is a well defined subgroup of $GL_1(2,K[t])$.

\begin{lemma}\label{generation} The group $GL_1(2,K[t])$ is generated by its subgroups $E_\gamma$, where $\gamma$ runs over $\P^1_K$
\end{lemma}

\begin{proof} 

For $A$, $B\in\End(K^2)$, set

\centerline{$<A\vert \,B>=\det\,(A+B) -\det\,A-\det\,B$.}

Any element $G(t)\in GL_1(2,K[t])$ can be written as a
polynomial

\centerline{$G(t)=\sum_{n\geq 0}\, A_n t^n$}

\noindent where $A_n$ belong to $\End(K^2)$ and $A_0=\id$. 

The proof runs by induction on the degree $N$ of $G$. 
It can be assumed that $N\geq 1$. For any $n\geq 0$, let
$I_n$ be the set of pairs of integers $(i,j)$
with $0\leq i<j$ and $i+j=n$. We have
$\det\,G(t)=\sum_{n\geq 0}\, c_n t^n$, where the scalars $c_n$ are given by

\centerline{$c_n=\det\,A_{n/2} + \sum_{(i,j)\in I_n}\, <A_i\vert A_j>$,}

\noindent where it is understood that  $\det\,A_{n/2}=0$ if $n$ is odd.

Since $\det\,G(t)$ is an invertible polynomial, we have $c_n=0$ for any $n>0$.
The identity $c_{2N}=0$ implies that $\det\,A_N=0$ hence $A_N$ has rank one.
Let $\delta$ be the image of $\,A_N$ and 
set $E=\{a\in\End(K^2)\vert\,\Image\,a\subset\delta\}$

There exist an integer $n$ such that

\centerline{$A_n\notin E$,
but $A_m\in E$ for any $m>n$.}

We have

\centerline{$c_{N+n}=\det\,A_{N+n\over 2} +\sum_{(i,j)\in I_{N+n}}\,<A_i\vert A_j>$}

\noindent Since $E$ consists of rank one endomorphisms, we have 
$\det a =0$ for any $a\in E$ and $<a,b>=0$ for any $a,\,b\in E$.
Therefore we have $\det\,A_{N+n\over 2}=0$ and $<A_i\vert A_j>=0$, whenever
$(i,j)$ lies in $I_{N+n}$ and $i\neq n$. Thus it follows that
$<A_n\vert \,A_N>=0$.

Set $\delta'=\Ker A_N$.  An easy computation shows that the previous relation 
$<A_N\vert \,A_n>=0$ implies
that $A_n(\delta')\subset \delta$. Set $B=e_\delta\circ A_n$. 
Clearly we have $\Ker \,B\supset\delta'$ and $\Image\,B\subset\delta$, therefore
$B$ is proportional to $A_N$. Since $B$ is not in $E$, we have $B\neq 0$ and therefore
$ce_\delta\circ A_n=A_N$ for some $c\in K$.
Set 

\centerline{$H(t)= (1-ct^{N-n} e_\delta).G(t)$.}

\noindent We have $e_{\delta}.a=0$ for any $a\in E$, therefore we have
$e_\delta.A_m=0$ for any $m>n$. It follows that 
$H(t)$ has degree $\leq N$. Moreover, its degree $N$ component
is $A_N-ce_\gamma.A_n=0$. Therefore $H(t)$ has degree
$<N$, et the proof runs by induction.
\end{proof}

\bigskip
\noindent
{\it 6.4 Free products in $GL_1(K[t]$}

\noindent Let $K$ be a field. 

\begin{lemma}\label{Tits}  We have 

\centerline{$GL_1(2,K[t])= *_{\delta\in\P^1_K}\,E_\delta$.} 
\end{lemma}

\begin{proof} Set $\Omega:=K[t]^2\setminus\{0\}$. Any $v\in\Omega$
can be written as a finite sum $v=\sum_{0\leq k}\,v_k\otimes t^k$,
where  $v_k$ lies in $K^2$. The biggest integer $n$ with 
$v_n\neq 0$ is the {\it degree} of $v$ and 
$hc(v):=v_n$ is called its highest component.
For any $\delta\in \P^1_K$, set 
$E^*_\delta=E_\delta\setminus\{1\}$ and

\centerline
{$\Omega_\delta=\{v\in \Omega\vert\,\, hc(v)\in\delta\}$.}

Let $H(t)\in E^*_\delta$. Since $H(t)\neq 1$, its degree $n$
is positive and its degree $n$ component is $c.e_\delta$ for some
$c\neq 0$. Let $\delta'\in\P^1_K$ be a line disctinct 
from $\delta$
and let $v\in\Omega_{\delta'}$. Since $e_\delta.u\neq 0$ for
any non-zero $u\in\delta'$, we have

\centerline{$hc(H(t).v)=ce_\delta hc(v)$,}

\noindent for any $v\in \Omega_{\delta'}$. It follows that
$E^*_\delta.\Omega_{\delta'}\subset\Omega_\delta$ for any
$\delta\neq\delta'$.

Therefore by Lemma \ref{Ping-Pong}, the free product
$*_{\delta\in\P^1_K}\,E_\delta$ embeds in $GL_1(2,K[t])$.
Hence by Lemma \ref{generation}, we have
$*_{\delta\in\P^1_K}\,E_\delta=GL_1(2,K[t])$.

\end{proof}

\bigskip

\noindent {\it Remark.}  
Set

$U^+=\{\begin{pmatrix} 1& f\\ 0&1 \end{pmatrix}
\vert\,\textnormal{ for}\, f\in K[t]\},$

$U^-=\{\begin{pmatrix} 1& 0\\ f&1 \end{pmatrix}
\vert\,\textnormal{ for}\, f\in tK[t]\},$ and

$U=\{\begin{pmatrix} 1& x\\ 0&1 \end{pmatrix}
\vert\,\textnormal{ for}\, x\in K\}$.

It is easy to show that Lemma \ref{Tits} is equivalent to
the fact that $U^+*U^-=GL_U(2,K)$. This results in stated in the context of Kac-Moody groups in Tits notes \cite{T82}.
Since these notes are not widely distributed, let mention that an equivalent result is stated in \cite{T89}, Section 3.2 and 3.2. For Tits original proofs, see \cite{T87}.

\section{The Linear Representation of  
$\Aut_{1}\,K^2$}

In this section, 
we prove Theorem B.

\noindent
{\it 7.1 The subgroups $F_\delta$ in $\Aut_{1}\,K^2$.}

\noindent
Let $\delta\in\P^1_K$ and let $(a,b)\in\delta$ be nonzero.
For any $f\in  K[t]$, let $\tau_\delta(f)$ be the 
automorphism

\centerline {$\tau_\delta(f): 
(x,y)\to (x+af(bx-ay), y+bf(bx-ay))$.} 

\noindent We have 
$\tau_\delta(f)\circ \tau_\delta(g)=\tau_\delta(f+g)$. Let
$F_\delta =\{\delta(f)\vert \,f\in t^2 K[t]\}$. Indeed for
$\delta_0=K(0,1)$, $\tau_{\delta_0}(f)$ is the elementary
automorphism

\centerline{$(x,y)\mapsto (x,y+f(x))$,}

\noindent and  $F_{\delta_0}$ is the group $\Elem_1(K^2)$ of Section 1.5.  In general, we have
$F_\delta=\Elem_1(K^2)^g$, where $g\in GL(K^2)$ satisfies
$g.\delta_0=\delta$.

\begin{lemma}\label{freeAut}  We have

\centerline{$\Aut_1(K^2)= *_{\delta\in\P^1_K}\,F_\delta$.}

\end{lemma}

\begin{proof} 
First we check that $\Aut_0\,K^2$ satisfies the hypotheses of Lemma \ref{mixing1}. By van der Kulk Theorem, we have 

\centerline{$\Aut_0\,K^2=G_1*_A\,G_2=\Gamma_0$}

\noindent where $G_1=GL(K^2)$, $G_2=\Elem_0(K^2)$ and
$A=B_0$ is the Borel subgroup of $GL(K^2)$. We have
$G_2=A\ltimes E$ where $E=\Elem_1(K^2)$. Clearly the
map $\chi:G_1*_A\,G_2\to G_1$ is simply the map

\centerline{$\phi\in\Aut_0(K^2)\to \textnormal{d}\phi_{\bf 0}$,}

\noindent and its kernel is $\Gamma_1=\Aut_1(K^2)$.
Since $G_1/A=GL(2)/B_0=\P^1_K$, and 
$E^\gamma=F_\delta$ for any $\gamma\in GL(K^2)$ with
$\gamma.\delta_0=\delta$,
it follows from
Lemma \ref{mixing1} that

\centerline{$\Aut_1(K^2)=*_{\gamma\in\P^1_K}\,F_\delta$.}

\end{proof}

\bigskip
\noindent{\it 7.2 Proof of Theorem B}

\noindent For any $\delta\in\P^1_K$, let 
$\psi_\delta:F_\delta\to E_\delta$ be the isomorphism
defined  by

\centerline{$\psi_\delta(\tau_\delta(f))=\id +f/t\otimes e_\delta$.}

\begin{MainB} The collection of isomorphisms 
$(\psi_\delta)_{\delta_{\P^1_K}}$ induces an isomorphism

\centerline{$\psi:\Aut_1\,K^2\simeq GL_1(2,K[t])$.}

In particular, $\Aut_1\,K^2$ is linear over $K(t)$

\end{MainB}

\begin{proof} The first statement is a consequence of
Lemmas \ref{Tits} and \ref{freeAut}. It follows that
$\Aut_1\,K^2\subset SL(2,K(t))$, therefore
$\Aut_1\,K^2$ is linear over $K(t)$.
\end{proof}

\bigskip
\noindent{\it Remark} The isomorphism 
$\Aut_1\,K^2\simeq GL_1(2,K[t])$ is not canonical, since
it depends upon a choice, for each $\delta\in\P^1_P$, of a basis
of $\delta$. 

Moreover the formula defining $\psi$ is complicated.
Indeed if an automorphism $\sigma\in \Aut_1(K^2)$ is written
as a free product $\sigma=\sigma_1\dots\sigma_m$, where $\sigma_i\in F_{\delta_i}$ has degree $m_i$, then $\sigma$ has degree
$m_1\dots m_n$ although $\psi(\sigma)$ has degree
$m_1+\dots +m_m-m$.

\section{Linearity of $\Aut_S\,K^2$ for some $S\subset GL(2,K)$}

In this section, Theorem C1 and C2 are proved.

For a subgroup $S$ of $GL(2,K)$, 
the  following hypotheses $\cal U$ and $\cal{QU}$ will be considered

\noindent$(\cal U)$\hskip15mm any $K$-reducible element  $g\in S$ is unipotent,

\noindent$(\cal QU)$\hskip15mm any $K$-reducible element  $g\in S$ is quasi-unipotent.

\bigskip\noindent
{\it 8.1 Proof of Theorem C1}

\begin{MainC1} Assume that the subgroup $S$ satisfies the hypothesis $\cal U$. Then we have

\centerline{$\Aut_S\,K^2\simeq GL_S(2,K[t])$.}

In particular, $\Aut_S\,K^2$ is linear over $K(t)$.
\end{MainC1}

\begin{proof} 

It follows from Lemmas \ref{Tits} and \ref{freeAut}
that 

\centerline{$GL_1(2,K[t])=*_{\delta\in\P^1_K}\,E_\delta$ and
$\Aut_1\,K^2=*_{\delta\in\P^1_K}\,F_\delta$,}

\noindent where the groups $E_\delta$ and $F_\delta$ are obviously isomorphic. For  $\delta\in\P^1_K$, let $S_{\delta}$ be
the stabilizer in $S$ of $\delta$. Any element in $S_{\delta}$ is
$K$-reducible, so by hypothesis $(\cal U)$ all elements in
$S_\delta$ are unipotent. It follows that $S_\delta$ acts trivially on $E_\delta$ and $F_\delta$, therefore the $S_\delta$
groups $E_\delta$ and $F_\delta$ are $S_\delta$-isomorphic.
It follows from Lemma \ref{mixing3} that the $S$ groups
$*_{\delta\in\P^1_K}\,E_\delta$ and
$*_{\delta\in\P^1_K}\,F_\delta$ are $S$-isomorphic. Therefore
$GL_S(2,K[t])=S\ltimes (*_{\delta\in\P^1_K}\,E_\delta)$ and
$\Aut_1\,K^2=S\ltimes (*_{\delta\in\P^1_K}\,F_\delta)$ are isomorphic.
\end{proof}

\bigskip\noindent
{\it 8.2 FG subgroups of $GL(2,K)$}

\noindent 
\begin{lemma}\label{quasiorder}
Let $S$ be a FG subgroup of $GL(2,K)$. There is an integer
$N>0$ such that any  quasi-unipotent $g\in S$ has quasi-order divisible by $N$.
\end{lemma}

\begin{proof} Let $F$ be the prime field of $K$, namely
$F=\F_p$ if $K$ has characteristic $p$ and $F=\Q$ otherwise.
There is a FG ring $R\subset K$ such that
$S$ is a subgroup of $GL(2,R)$. It can be assumed that $K$
is the fraction field of $R$.

Let $X$ be the set of all roots of 
unity which are eigenvalues of some element $g\in S$. We claim that the set $X$ is finite.

Set $E=\overline F\cap K$. Since $K$ is finitely generated,
$E$ is a finite extension of $F$. Any $\zeta\in X$ lies in $E$ or in a quadratic
extension of $E$. We will now consider separately the case where $K$ has zero characteristic or finite characteristic.

First assume that $\ch\,K=0$.  Let $m$ be the order of  $\zeta$. Since $[\Q(z):\Q]=\phi(m)$, we have $\phi(m)\leq 2[E:\Q]$. Thus 
$m$ is bounded, and therefore $X$ is finite.

Next asssume that $\ch\,K=p$. Then $E$ is a finite field and 
$\zeta$ lies in the unique quadratic extension $E'$ of $E$.
Therefore $X\subset E'$ is finite.

Since $X$ is finite, we have $X\subset \mu_N$ for some $N$,
where $\mu_N$ is the set of $N^{th}$-root of unity. Therefore
any quasi-unipotent element $g\in S$ has eigenvalues in $\mu_N$, and therefore $g^N$ is unipotent.

\end{proof}

\begin{lemma} \label{quasiU}
Let $S$ be a FG subgroup of $GL(2,K)$ satisfying
the hypothesis $\cal {QU}$. Then $S$ contains a finite
index subgroup $S'$ satisfying the hypothesis
${\cal U}$.
\end{lemma}

\begin{proof} By Lemma \ref{quasiorder}, there exists 
an integer $N\geq 1$ such that any $K$-reducible element in $S$
is quasi-idempotent of quasi-order divisible by $N$.

Let $C$ be the set of pairs $(s,p)$ in $K^2$  such that
$s=\zeta_1+\zeta_2$ and $p=\zeta_1 \zeta_2$ for some
$\zeta_1,\zeta_2\in\mu_N$.

There is a finitely generated subring $R$ of $K$ such that
$S\subset GL(2,R)$. Since the intersection of all cofinite ideal $m$ in $R$ is zero, there is a cofinite ideal $m\subset R$
such that $(s,p)\not\equiv (2,1)\mod m$ for any $(s,p)\in C\cap R^2$
with $(s,p)\neq (2,1)$.

Set $S'=S\cap GL(2,m)$. Since $GL(2,m)$ has finite index in $GL(2,R)$, $S'$ has finite index in $S$. We claim that
$S'$ satisfies the hypothesis ${\cal U}$. 

Let $g\in S$' be
$K$-reducible. Since $g$ is quasi-idempotent of quasi-order
divible by $N$, the couple $(\tr\,g,\det\,g)$ belongs to
$C$. Since $g\equiv 1\mod m$, we have 
$(\tr\,g,\det\,g)\equiv (2,1)\mod m$. By the construction of $m$,
we have $(\tr\,g,\det\,g)=(2,1)$, and therefore $g$ is
unipotent. Hence $S'$ satisfies the hypothesis $({\cal U})$.

\end{proof}

\bigskip\noindent
{\it 8.3 Proof of Theorem C.2}

\bigskip

\begin{MainC2} Let $S$ be a FG subgroup of $GL(K^2)$
satisfying hypothesis $({\cal QU})$. Then 
$\Aut_S\,K^2$ is linear over $K(t)$.

\end{MainC2}

\begin{proof} By Lemma \ref{quasiU}, there is a subgroup $S'$
of $S$ of finite index satisfying the hypothesis $(\cal U)$.
By Theorem C.1, the group $\Aut_{S'}\,K^2$ is linear over $K(t)$.
Since it is a finite index subgroup, the group 
$\Aut_S\,K^2$ is also linear over $K(t)$.
\end{proof}

\section{Some Corollaries}

{\it 9.1 Linearity of $\Aut\,K^2$ for a finite field $K$}

\begin{CorD1} Let $K$ be a finite field of characteristic $p$. The
group $\Aut\,K^2$ is linear over $\F_p(t)$.

\end{CorD1}
 
 \begin{proof} By Theorem B, the group
 $\Aut_1\,K^2$ is  Since 
  $[\Aut\,K^2:\Aut_1\,K^2]$ is finite, the group
 $\Aut\,K^2$ is also  linear over $\F_p(t)$.
 
\end{proof}

\bigskip\noindent
{\it 9.2 Linearity of FG subgroups of $\Aut\,K^2$ for
a quasi-finite field $K$}

\begin{CorD2} Let $K$ be an infinite subfield of
$\overline F_p$.

The group $\Aut\,K^2$ is not linear, even over a ring.
However any FG subgroup is linear over
$F_p(t)$

\end{CorD2}

\begin{proof} Indeed $\Aut\,K^2$ is not linear by
Theorem B. Any FG subgroup $\Gamma$ of $\Aut\,K^2$  is a subgroup
of $\Aut\,L^{2}$ for some finite subfield $L\subset K$, so it is linear by Corollary D.1.
\end{proof}

\bigskip\noindent
{\it 9.3 Linearity over $K$ of $\Aut\,K^2$ for
a big enough field $K$}

\begin{lemma}\label{Card}
Let $K$, $L$ be  fields. If
$\Card\, K=\Card\, L$ and $\ch\, K=\ch\, L$,
then $\Aut_1\,K^2$ and $\Aut_1\,L^2$
are isomorphic.

\end{lemma}

\begin{proof} It can be assumed that $K$ is infinite.
Let $F$ be its prime field, let $P$ be a set and let
$E$ be a $F$-vector space with

\centerline{$\Card P=\Card K$ and $\dim_F\,E=\Card\,K$.}

By lemma \ref{freeAut}, we have

\centerline{$\Aut_1\,K^2\simeq *_{p\in P}\, E_p$,}

\noindent where each $E_p$ is a copy of $E$. It follows that 
 $\Aut_1\,K^2$ and $\Aut_1\,L^2$
are isomorphic if $\Card\, K=\Card\, L$ and $\ch\, K=\ch\, L$.
\end{proof}

A field $K$ is called {\it big enough} if its
absolute transcendence degree is
$\geq 1$ and if $K$ is not a finite extension of
 $\F_p(t)$ for some prime number $p$. For example,
 $\Q(t),\,\F_p(t_1,t_2)$ and $L(t)$,
where $L$ is an infinite subfield of $\overline F_p$,  are big enough.

\begin{CorD} Let $K$ be a big enough field. There is
an embedding

\centerline{$\Aut_1\,K^2\subset SL(2,K)$.}
\end{CorD}

\begin{proof} Since $K$ is big enough, there is an
embedding $L(t)\subset K$, where $L$ is a field 
with $\Card \,L=\Card\,K$. 

It follows from Lemma \ref{Card} that 
$\Aut_1\,K^2$ is isomorphic to 
$\Aut_1\,L^2$. Moreover by Theorem B, 
$\Aut_1\,L^2$ is isomorphic to
$GL_1(L[t])$. Therefore we have

\centerline{
$\Aut_1\,K^2\simeq GL_1(L[t])
\subset SL_2(L(t))\subset SL(2,K)$.}

\end{proof}

\end{document}